\documentclass[12pt,reqno]{amsart}
\usepackage{amsmath}
\usepackage{amssymb}
\usepackage{graphicx}
\usepackage[all]{xy}
\newtheorem{thm}{Theorem}[section]

\newtheorem{cor}{Corollary}[section]
\newtheorem{lem}{Lemma}[section]

\newtheorem{prop}{Proposition}[section]
\newtheorem{exa}{Example}

\theoremstyle{definition}
\newtheorem{defn}{Definition}[section]
\newtheorem{rem}{Remark}[section]
\newcommand{\be}{\begin{equation}}
\newcommand{\ee}{\end{equation}}
\newcommand\bes{\begin{eqnarray}}
 \newcommand\ees{\end{eqnarray}}
\newcommand{\bess}{\begin{eqnarray*}}
\newcommand{\eess}{\end{eqnarray*}}

 \numberwithin{equation}{section}

\begin{document}
\title{ On propositional logic semirings }

\maketitle
\begin{center}
{\sl  Wenxi LI\footnote{  E-mail address: wxli@ahut.edu.cn.}  \ Zhongzhi WANG\footnote{ E-mail address: zhongzhiw@126.com.} }

\

( School of Microelectronics and Data Science, AnHui University of
Technology, Ma'anshan, 243002, China)

\

\end{center}

 \indent {\bf Abstract:}
        {\small }Propositional logic serves as a fundamental cornerstone in mathematical logic. This paper delves into a semiring characterization of propositional logic, employing the Gr\"oebner-Shirshov basis theory to furnish an algebraic framework for deduction and proof grounded in atoms of propositional logic. The result is an algebraic approach to proving propositions in propositional logic. To illustrate the effectiveness and constraints of this method, we conclude with several specific examples.

\
{\bf{Key words:}} Propositional logic,\ \ Gr\"oebner-Shirshov bases,\ \ semirings

{\bf{Mathematics Subject Classification (2020):}}
03B05, 13P10, 16Y60.
\section{\bf Introduction}
In a letter to Nicolas Remnond\cite{Leibniz1969}, Leibniz expressed, '$\cdots$ if I had been less distracted or if I were younger or had talented young men to help me, I should still hope to create a kind of universal symbolistic in which all truths of reason would be
reduced to a kind of calculus. $\cdots$ this could be a kind of universal language or writing
 for the characters and the words themselves would give directions to reason, and the
errors -except those of fact- would be only mistakes in calculation.'  This statement is a key component of what is known as Leibniz's Program -a comprehensive initiative aimed at promoting peace and justice through the formalization or mechanization of reason\cite{AndersonZeleny2001}. Propositional calculus and predicate calculus are two fundamental formal systems that constitute this exciting program, serving as mathematical models for the study of reasoning.

However,  the conjunction and disjunction operations of propositions satisfy the requirements of a semiring operation. Therefore, the propositional logic system with the conjunction and disjunction operations forms a semiring, where each element in this semiring has a complement. In our previous paper, we introduced idempotent complement semirings. In this paper, we attempt to represent the propositional logic system using an idempotent complement semiring generated by countable elements. Our goal is to obtain an algebraic representation for proofs and deductions, and we employ the Gr\"obner-Shirshov basis theory to provide an algebraic characterization of propositions provable within a propositional system. We have developed a method for proving propositions through set calculations. In conclusion, we present several simple examples to illustrate and explain this method. Unfortunately, it appears that this approach does not significantly contribute to reducing the complexity of problem proofs.

We use $\mathbb{N}$($\mathbb{N}^+$) to  denote  the set of (positive) natural numbers. Concepts and symbols related to mathematical logic in this paper, which are not explicitly explained, can be referred to\cite{EFT2021}. Undefined definitions and conclusions regarding the Gr\"obner-Shirshov (abbreviated as G-S) bases for semirings in this paper can be found in the reference\cite{Bokut2013}.

\section{\bf   Gr\"{o}bner-Shirshov bases for finitly generated semirings  }
In a previous paper\cite{NLW}, we provided the G-S bases for finitely generated idempotent complement semirings. In this section, for the sake of readability, we list relevant concepts and conclusions without presenting the proofs.

\begin{defn}\label{Definition 2.1}\cite{Golan}
A semiring is a nonempty set $S$ on which operations of "$\circ $" and multiplication $"\cdot"$ have been defined such that the following conditions are satisfied:

1) $(S , \circ)$ is a commutative monoid with identity element $\theta$.

2) $(S , \cdot)$ is a monoid with identity element $1$.

3) "$\cdot$" is distributive relative to $\circ$ form left and right.

4) $\theta \cdot s = s \cdot \theta = \theta$,  for all  $s \in S$.

5) $1_{S} \neq \theta$.

If $(S , \cdot)$ is commutative, then $S$ is called a commutative semiring.
\end{defn}

Let $X=\{x_{1},\cdots ,x_{n}\}, X^{c}=\{x^{c}_{1},\cdots,x^{c}_{n}\}$, $\tilde{X} = X\cup X^{c}$ and $Rig[\tilde{X}]$ be the free commutative semiring generated by $\tilde{X}$. We define a congruence relation $\equiv_{\rho}$ in $Rig[\tilde{X}]$ generated by $\{(x_{i}\cdot x^{c}_{i},\theta),(x_{i}\circ x^{c}_{i},1),(x_{i}\circ x_{i},x_{i}),(x^{c}_{i}\circ x^{c}_{i},x^{c}_{i})\}.$  and  denote $Rig[\tilde{X}]/\equiv_{\rho}$ by $S_{c}[X]$.

To study the G-S basis of  $S_{c}[X]$, we  introduce a set $Y=\{y_{0},\cdots ,y_{m-1}\}(m=2^{n})$, where  $y_{k}=x^{k_{1}}_{1}\cdot \cdots \cdot x^{k_{n}}_{n} \in Rig[\tilde{X}]$  and
\be
x^{k_i}_{i}=\left\{
\begin{aligned}
x_{i}  \quad {k_i=1}\\
x^{c}_{i}  \quad {k_i=0}\\
\end{aligned}
\right
.
\ee
for $k=\sum^{n}\limits_{i=1}k_{i}2^{i-1}$, $k_{i}\in \{0,1\}$.

Denote  $\sum^{\circ}$  as the sum of $"\circ"$ operation, and
\be
\begin{aligned}
 A_{i}^{(n)}=\{\sum^{n}\limits_{j=1}k_{j}2^{j-1}|k_{i}=1, k_{1},\cdots,k_{i-1},k_{i+1},\cdots,k_{n}\in \{0,1\}\},
\end{aligned}
\ee
\be
\begin{aligned}
A^{c(n)}_{i}=\{\sum^{n}\limits_{j=1}k_{j}2^{j-1}|k_{i}=0, k_{1},\cdots,k_{i-1},k_{i+1},\cdots,k_{n}\in \{0,1\}\}.
\end{aligned}
\ee

A  monomial order in $Rig[\tilde{X}\cup Y]$ defined as following:

\begin{align*}
x_{1}>x_{2}>\cdots >x_{n}>x^{c}_{1}>x^{c}_{2}>\cdots >x^{c}_{n}>y_{0}>\cdots >y_{m-1}
\end{align*}

Then for  any element in $[\tilde{X}\cup Y]$ with form $u=a_{1}\cdot a_{2}\cdot \cdots \cdot a_{l}$ ($ a_{1}\geq a_{2}\geq \cdots \geq a_{l}$, and $u=1$ iff $l=0$),
we order $[\tilde{X}\cup Y]$ as follows: for any $a,b\in [\tilde{X}\cup Y]$, if one of the sequences is not a prefix of orther, then lexicographically; if the sequence of $a$ is a prefix of the sequence of $b$, then $a<b$.\\
\indent For any  $w=u_{1}\circ u_{2}\circ \cdots \circ u_{k}\in Rig[\tilde{X}\cup Y]$ where $u_{i} \in [\tilde{X}\cup Y]$ and$ u_{1}\geq u_{2}\geq \cdots \geq u_{k}$. We order $Rig[\tilde{X}\cup Y]$  lexicographically.
\begin{thm}\label{Theorem 3.1}\cite{NLW}
Let the ordering on $Rig[\tilde{X}\cup Y]$ be as above. Then $kRig[\tilde{X}\cup Y|\{(x_{i}\cdot x^{c}_{i},\theta),(x_{i}\circ x^{c}_{i},1),(x_{i}\circ x_{i},x_{i}),(x^{c}_{i}\circ x^{c}_{i},x^{c}_{i}),(y_{k},x^{k_{1}}_{1}\cdot \cdots \cdot x^{k_{n}}_{n})\}]= kRig[\tilde{X}\cup Y|R_{1}]$ and $R_{1}$ is a Gr\"{o}bner-Shirshov basis in $kRig[\tilde{X}\cup Y]$, where $R_{1}$ consists of the following relations:

$(r_{1})$. $y_{k}\cdot y_{k}=y_{k},$

$(r_{2})$. $y_{k}\cdot y_{j}=\theta$ $(k\neq j)$,

$(r_{3})$. $y_{k}\circ y_{k}=y_{k}$,

$(r_{4})$. $x_{i}=\mathop{\sum^{\circ}}\limits_{j\in A_{i}^{(n)}} y_{j}$ $(i\in\{1,2,\cdots ,n\})$, where $A_{i}^{(n)}$ is defined as in (2.2),

$(r_{5})$. $x^{c}_{i}=\mathop{\sum^{\circ}}\limits_{j\in A^{c(n)}_{i}}y_{j}$ $(i\in \{1,2,\cdots ,n\})$, where $A^{c(n)}_{i}$ is defined as in (2.3),

$(r_{6})$. $\mathop{{\sum}^{\circ}}\limits_{j\in D}y_{j}=1$,

$(r_{7})$. $1\circ y_{k}=1$,

$(r_{8})$. $1\circ 1=1$.
\end{thm}

Therefore a  normal form of the semiring $S_{c}[X]$ is the set
$$\{1, {\sum\limits_{k\in D}}^{\circ}y_{k}\},$$
where $D\varsubsetneq \{0,1,\cdots ,m-1\}.$

\begin{defn}\label{Definition 3.1}\cite{NLW}
Let $S$ be a comutative semiring with operations $"\circ"$ and multiplication $"\cdot"$. If for any $s\in S$, there exists $t\in S$ such that the following conditions are satisfied:

(1) $s\cdot t=\theta$.

(2) $s\circ t=1$.

(3) $s\cdot s=s$.

(4) $s\circ s=s$.

Then we call $S$ is an idempotent complement(I-C) semiring, where $t$ is the complement of $s$, and denote $t=s^{c}$.
\end{defn}

From the normal form of $S_{c}[X]$, it is easy to prove that $S_{c}[X]$ is an I-C semiring, and it is a free idempotent complement semiring generated by $X$\cite{NLW}.

\section{\bf Idempotent complement semirings}
Typically, a propositional logic system is composed of countably many atoms. Therefore, in this section, we consider  countably generated free I-C semirings.

Suppose $X=\{x_{1},\cdots ,x_{n},\cdots\}$ be a countable set,  let $X^c=\{x_{1}^c,\cdots ,x_{n}^c,\cdots\}$, $\tilde{X}=X\cup X^c$, and $X_{N}=\{x_{1},\cdots,x_{N}\}$ for any $N\in \mathbb{N}^+$. Denote $S_c[X_N]$ the free I-C semiring generated by $X_N$ .

Denote $Y=\{y_{0}^{(1)},y_{1}^{(1)}\cdots,y_{0}^{(n)},\cdots,y_{m}^{(n)},\cdots\}$, where $m=2^{n}-1, n\in \mathbb{N}^+$,  and for each $k=\sum^{n}\limits_{l=1}j_{l}2^{l-1}$, $j_{l}\in \{0,1\}$ and $y_{k}^{(n)}=x^{j_{1}}_{1}\cdot \cdots \cdot x^{j_{n}}_{n} \in Rig[\tilde{X}]$, where \\
\be
x^{j}_{i}=\left\{
\begin{aligned}
x_{i}  \quad {j=1}\\
x^{c}_{i}  \quad {j=0}\\
\end{aligned}
\right
.
\ee
We first give a direct system  $(S_c[X_N],f_{NM})$  as following:

For any $N\leq M$, there exists a natural embedding  monomorphism $i_{NM}$ from $Rig[X_N\cup X_N^c]$ to $Rig[X_M\cup X_M^c]$, where $X_{N}^c=\{x_{1}^c,\cdots,x_{N}^c\}$ and  $Rig[X_N\cup X_N^c]$ the
free commutative semiring generated by $X_N\cup X_N^c$.

Since $S_c[X_k]=Rig[X_k\cup X_k^c]/\equiv_k$, where $\equiv_k$ the congruence relation in $Rig[X_k\cup X_k^c]$ generated by $\{(x_i\cdot x_i^c,\theta),(x_i\circ x_i^c,1),(x_i\circ x_i,x_i),(x_i^c\circ x_i^c,x_i^c)|i=1,2,\cdots,k\}$. Obviously, $\equiv_N\subseteq\equiv_M$ for any $N\leq M$.

Let $\eta_k:\  Rig[X_k\cup X_k^c]\rightarrow S_c[X_k]$ be the natural epimorphism, for any $\alpha,\beta\in Rig[X_N\cup X_N^c]$, if $\alpha\equiv_N\beta$, since $\equiv_N\subseteq\equiv_M$, so $\alpha\equiv_M\beta$, this means that $\eta_M i_{NM}(\alpha)=\eta_M i_{NM}(\beta)$, so there exists an unique homomorphism $f_{NM}$, such that the following diagram commutative:
\begin{displaymath}
\xymatrix{ Rig[X_N\cup X_N^c]\ar[d]_{\eta_N} \ar[r]_{\ \ \ \ \ \eta_M i_{NM}} & S_{c}[X_M]\ar@{<--}[dl]^{f_{NM}}\\
            S_c[X_N] &}
\end{displaymath}

\begin{prop}\label{fnmmonic}
For any $N\leq M$, $f_{NM}$ is a monomorphism.
\end{prop}
\begin{proof}
For any $\alpha,\beta\in S_c[X_N]$, suppose $\alpha=\mathop{\sum^{\circ}}\limits_{i\in A} y_{i}^{(N)},\beta=\mathop{\sum^{\circ}}\limits_{j\in B} y_{j}^{(N)}$ be the normal forms of $\alpha,\beta$ in $S_c[X_N]$. Using  the G-S basis for $S_c[X_M]$, it is straightforward to obtain the normal form of  $f_{NM}(y_k^{(N)})=\mathop{\sum^{\circ}}\limits_{l \mod2^N=k}^{l\leq2^M-1} y_{l}^{(M)}$ in $S_c[X_M]$, Consequently,  the normal form of  $f_{NM}(\alpha)=\mathop{\sum^{\circ}}\limits_{k\in A}\mathop{\sum^{\circ}}\limits_{l\mod2^N=k}^{l\leq2^M-1} y_{l}^{(M)}$. If $f_{NM}(\alpha)=f_{NM}(\beta)$, examining the index values less than  $2^N$ in the  normal forms of $f_{NM}(\alpha)$ and $f_{NM}(\beta)$ in $S_c[X_M]$, allows us to conclude that $A=B$. Therefore $\alpha=\beta$, $f_{NM}$ is a monomorphism.

\end{proof}

Through routine verification, we establish the following proposition:
\begin{prop}
For any $N \leq M \leq K$, it holds that $f_{NK} = f_{MK} \cdot f_{NM}$, and $f_{NN}$ is an identity homomorphism.
\end{prop}

The triple $\{S_c[X_N], f_{NM}, \mathbb{N}^+\}$ forms a direct system. We can construct its direct limit $S_c[X]$ as follows:

Let $S=\bigsqcup\limits_{N\in\mathbb{N}^+}({S_c[X_N]\times\{N\}})$, and define an equivalent relation in $S$ as $(a,i)\sim(b,j)$ iff $ \exists k\geq i, k\geq j$, such that $f_{ki}(a)=f_{kj}(b)$. Denote $S_c[X]=S/\sim$, and define two operations in it as $\overline{(a,i)}\circ\overline{(b,j)}=\overline{(f_{ki}(a)\circ f_{kj}(b),k)},\overline{(a,i)}\cdot\overline{(b,j)}=\overline{(f_{ki}(a)\cdot f_{kj}(b),k)}$ for some $k\geq i,j$. It is not difficult to verify that $S_c[X]$ is a semiring and $S_c[X]=\lim\limits_{\rightarrow}S_c[X_N]$.

\begin{prop}
$S_c[X] $ is the free I-C semiring generated by $X$.
\end{prop}
\begin{proof}

Since each $S_c[X_N]$ is an I-C semiring, it is easy to see that $S_c[X]$ is an I-C semiring.

Denote $\eta: X \cup X^c \longrightarrow S_c[X]$ via $x_i \mapsto \overline{(x_i, i)}$ and $x_i^c \mapsto \overline{(x_i^c, i)}$. If $\eta(x_i) = \eta(x_j)$, then $\overline{(x_i, i)} = \overline{(x_j, j)}$, so there exists $k \geq i, j$ such that $f_{ki}(x_i) = f_{kj}(x_j)$ in $S_c[X_k]$. By using the G-S basis for $S_c[X_k]$, we can conclude that $x_i = x_j$, and $x_i \neq x_j^c$ for any $i, j$ similarly. Thus, $\eta$ is a monomorphism.

For any I-C semiring $S$ and a map $f: X \cup X^c \longrightarrow S$ satisfying $f(x_i) \cdot f(x_i^c) = \theta$ and $f(x_i) \circ f(x_i^c) = 1$, define $\bar{f}: S_c[X] \longrightarrow S$ via $\overline{(\alpha, i)} \mapsto f(U_1) \circ f(U_2) \cdots \circ f(U_n)$, where $\alpha = U_1 \circ U_2 \cdots \circ U_n$, $U_k = x_1^{k_1} \cdot x_2^{k_2} \cdots x_i^{k_i}$, and $f(U_k) = f(x_1^{k_1}) \cdot f(x_2^{k_2}) \cdots f(x_i^{k_i})$. Here, $x_l^h$ denotes $1, x_l, x_l^c$ when $h = 0, 1, c$.

It is easy to verify that $\bar{f}$ is a homomorphism, and the following diagram is commutative:
\begin{displaymath}
\xymatrix{ X\cup X^c\ar[d]_{f} \ar[r]_{\eta} & S_{c}[X]\ar@{-->}[dl]^{\bar{f}}\\
            S &}
\end{displaymath}
$\bar{f}$ is unique obviously, so $S_c[X]$ is the free I-C semiring generated by $X$.

\end{proof}
\section{\bf propositional logical semirings}
In this section, we use I-C semirings to represent  propositional logic system, and define concepts such as implication and deduction.

We use $Hom(S, T)$ to denote the set of homomorphisms from semiring $S$ to semiring $T$, and $End(S)$ to denote the set of endomorphisms of semiring $S$.

\begin{lem}\label{rehom}
Let $S$ be a semiring and $\tau\in End(S)$, $\equiv_1=<(a_1,b_1),\cdots,(a_n,b_n)>$ and $\equiv_2=<(\tau(a_1),\tau(b_1)),\cdots,(\tau(a_n),\tau(b_n))>$ be two congruence relations in $S$, then for any $x\equiv_1 y$,   $\tau(x)\equiv_2 \tau(y)$
\end{lem}
\begin{proof}
Let $g=\eta\tau$, where $\eta$ the natural epimorphism from $S$ to $S/\equiv_2$, then a congruence relation $\equiv_g$ in $S$ can be induced from $g$ as $x\equiv_g y$ iff $g(x)=g(y)$. For any $(a_i,b_i)$, since $\tau(a_i)\equiv_2 \tau(b_i)$, $g(a_i)=\eta(\tau(a_i))=\eta(\tau(b_i))=g(b_i)$, so $a_i\equiv_g b_i$,
$(a_i,b_i)\in \equiv_g$, $\equiv_1\subseteq\equiv_g$, then there exists an $f\in Hom(S/\equiv_1,S/\equiv_2)$, such that the following diagram commutative:
\begin{displaymath}
\xymatrix{ S\ar[d]_{\eta}\ar[r]_{g\ \ \ \ } & S/\equiv_2\ar@{<--}[dl]^{f}\\
            S/\equiv_1 &}
\end{displaymath}
Obviously, $f(\bar{x})=\overline{\tau(x)}$, so for any $x\equiv_1 y$, $\tau(x)\equiv_2 \tau(y)$
\end{proof}

\begin{lem}\label{homequiv}
Let $S$ be a semiring and $\tau\in End(S)$, $\equiv=<(a_1,b_1),\cdots,(a_n,b_n)>$  be a congruence relations in $S$, if $\tau(a_i)\equiv \tau(b_i)$ for $1\leq i\leq n$,  then $\tau(x)\equiv \tau(y)$ for any $x\equiv y$.
\end{lem}
\begin{proof}
Since $\equiv_1=<(\tau(a_1),\tau(b_1)),\cdots,(\tau(a_n),\tau(b_n))>\subseteq\equiv$ by $\tau(a_i)\equiv \tau(b_i)$ for $1\leq i\leq n$, and for any $x\equiv y$, $\tau(x)\equiv_1 \tau(y)$ by Lemma \ref{rehom}, so $\tau(x)\equiv \tau(y)$.
\end{proof}
\begin{lem}\label{homexist}
Let $\aleph_N=\{0,1,2,\cdots,2^N-1\}$, $\pi$ be a permutation on $\aleph_N$, then there is a $\tau\in End(S_c[X])$, such that $\tau(\overline{(y_i^{(N)},N)})=\overline{(y_{\pi(i)}^{(N)},N)}$ for $i \in \aleph_N$.
\end{lem}
\begin{proof}
For any $M\geq N$, we  define   $\tau_M:\ S_c[X_M]\rightarrow S_c[X_M]$ via

$\tau_M(\mathop{\sum^{\circ}}\limits_{i\in A} y_{i}^{(M)})=\mathop{\sum^{\circ}}\limits_{i\in A} y_{\pi(i\mod2^N)+i-i\mod2^N}^{(M)}$

It can be routinely verified that $\tau_M\in End(S_c[X_M])$ and $f_{MK}\tau_M=\tau_Kf_{MK}$ for any $N\leq M\leq K$.

 Then for any $\overline{(\alpha,i)}\in S_c[X]$,
we define $\tau(\overline{(\alpha,i)})=\overline{(\tau_Kf_{iK}(\alpha),K)}$ for some $K\geq i,\ N$.

If $K_1,K_2\geq i,N$, there exists a $K\geq K_1,K_2$, such that $f_{K_1K}\tau_{K_1}f_{iK_1}(\alpha)=\tau_{K}f_{K_1K}f_{iK_1}(\alpha)=\tau_{K}f_{iK}(\alpha)=\tau_{K}f_{K_2K}f_{iK_2}(\alpha)=f_{K_2K}\tau_{K_2}f_{iK_2}(\alpha)$, $\overline{(\tau_{K_1}f_{iK_1}(\alpha),K_1)}=\overline{(\tau_{K_2}f_{iK_2}(\alpha),K_2)}$, otherwise if
$\overline{(\alpha,i)}=\overline{(\beta,j)}\in S_c[X]$, there is a $K\geq i,j,N$, such that $f_{iK}(\alpha)=f_{jK}(\beta)$, $\tau(\overline{(\alpha,i)})=\overline{(\tau_K(f_{iK}(\alpha)),K)}=\overline{(\tau_K(f_{jK}(\beta)),K)}=\tau(\overline{(\beta,j)})$. So $\tau$ is well defined.

For $\diamond\in \{\cdot,\circ\}$, $\overline{(\alpha,i)}\diamond\overline{(\beta,j)}=\overline{(f_{iK}(\alpha)\diamond f_{jK}(\beta),K)}$ for some $K\geq i,j$, without loss of generality, let $K\geq N$, then  $\tau(\overline{(\alpha,i)}\diamond\overline{(\beta,j)})=\tau(\overline{(f_{iK}(\alpha)\diamond f_{jK}(\beta),K)})=\overline{(\tau_K(f_{iK}(\alpha)\diamond f_{jK}(\beta)),K)}=\overline{(\tau_Kf_{iK}(\alpha)\diamond \tau_Kf_{jK}(\beta),K)}=\overline{(\tau_Kf_{iK}(\alpha),K)}\diamond\overline{(\tau_Kf_{jK}(\beta),K)}=\tau(\overline{(\alpha,i)})\diamond\tau(\overline{(\beta,j)})$

Therefore  $\tau\in End(S_c[X])$, and $\tau(\overline{(y_i^{(N)},N)})=\overline{(\tau_N(y_i^{(N)}),N)}=\overline{(y_{\pi(i)}^{(N)},N)}$ for $i \in \aleph_N$.
\end{proof}

\begin{defn}\label{plsemiringdef}

Let $X=\{x_1,x_2,\cdots,x_n,\cdots\}$ represent a set of atoms in propositional logic, and $S_c[X]$ denote the free I-C semiring generated by $X$. Here, $x_i^c$ signifies the negation of $x_i$, and the operations $x_i \cdot x_j$ and $x_i \circ x_j$ correspond to logical AND ($x_i \wedge x_j$) and logical OR ($x_i \vee x_j$), respectively. In this context, we refer to $S_c[X]$ as a propositional logic semiring.
\end{defn}

\begin{defn}\label{implydef}
Let $X=\{x_1,x_2,\cdots,x_n,\cdots\}$ and $S_c[X]$ be the free I-C semiring generated by $X$. For any $\alpha, \beta \in S_c[X]$, the expression $\alpha^c \circ \beta$ is referred to as the implication that $\alpha$ implies $\beta$, and it is denoted as $\alpha \rightarrow \beta$.
\end{defn}

\begin{defn}\label{deducedef}
Let $X=\{x_1,x_2,\cdots,x_n,\cdots\}$ and $S_c[X]$ be the free I-C semiring generated by $X$. Suppose $\Gamma\subseteq S_c[X]$, $\beta\in S_c[X]$ and $\equiv_{\Gamma}=<\{(\alpha,1)|\alpha\in\Gamma\}>$ the congruence relation in $S_c[X]$ generated by $\alpha=1$ for all $\alpha\in \Gamma$. If $\beta\equiv_{\Gamma}1$, then we call $\beta$   deduced by $\Gamma$, denote as $\Gamma\vdash \beta$. Especially, if $\Gamma=\{\alpha\}$,  we denote it as $\alpha\vdash \beta$.

\end{defn}

\begin{prop}\label{deduced}
Let $X=\{x_1,x_2,\cdots,x_n,\cdots\}$,  $S_c[X]$ be the free I-C semiring generated by $X$, $\Gamma\subseteq S_c[X]$ and $\beta\in S_c[X]$. If $\Gamma\vdash \beta$, then there  exist $\alpha_1, \cdots, \alpha_n\in \Gamma$ such that $\alpha_1\cdot\alpha_2 \cdots\cdot\alpha_n\vdash\beta$.
\end{prop}
\begin{proof}
Suppose $\Gamma\subseteq S_c[X]$, and $\Gamma\vdash \beta$. Then $\beta-1$ belongs to the ideal generated by $\{\alpha-1|\alpha\in\Gamma\}$ in the semiring algebra $kS_c[X]$.  Thus,  there exist $\alpha_1, \cdots, \alpha_n\in \Gamma$ such that $\beta-1$ can be represented by $\alpha_1-1, \cdots, \alpha_n-1$. Therefore, $\{\alpha_1,\alpha_2,\cdots,\alpha_n\}\vdash\beta$. Utilizing the normal form of $S_c[X_N]$ in Theorem \ref{Theorem 3.1} for a big enough $N\in\mathbb{N}$, this is equivalent to  $\alpha_1\cdot\alpha_2 \cdots\cdot\alpha_n\vdash\beta$.
\end{proof}

From the properties of the G-S basis, we can easily obtain the following conclusion:

\begin{prop}\label{finiteimplication}

Let $X=\{x_1,x_2,\cdots,x_n\}$ be a finite set, and $S_c[X]$ be the free I-C semiring generated by $X$. Suppose $\Gamma\subseteq S_c[X]$, and let $R$ be a G-S basis of the ideal in $kS_c[X]$ generated by $\{\alpha-1|\alpha\in\Gamma\}$. Then, for any $\beta\in S_c[X]$, $\Gamma\vdash \beta$ if and only if $\beta-1$ is trivial modulo $R$.
\end{prop}

\begin{prop}\label{implication}
Let $X=\{x_1,x_2,\cdots,x_n,\cdots\}$ and $S_c[X]$ be the free I-C semiring generated by $X$. Suppose that  $\alpha=\overline{(\mathop{\sum^{\circ}}\limits_{i\in A} y_{i}^{(N)},N)}, \beta=\overline{(\mathop{\sum^{\circ}}\limits_{j\in B} y_{j}^{(N)},N)}$, then $\alpha\vdash \beta \Leftrightarrow A\subseteq B$.
\end{prop}
\begin{proof}
($\Rightarrow$). Let  $\equiv_{\alpha}$ be a congruence relation generated by $\alpha=1$ in $S_c[X]$, if $A\neq\emptyset$, then $\alpha\neq\theta$. It is easy to prove that $\equiv_{\alpha}$  is nontrivial, given that  $\equiv_{\alpha}\subseteq \equiv=<\{\alpha=1,\overline{(x_i,i)}=\overline{(x_{i+kN},i+kN)},\overline{(x_i^c,i)}=\overline{(x_{i+kN}^c,i+kN)}|i=1,2,\cdots,k=0,1,\cdots\}>$

Assume that $A\nsubseteqq B$. Then $A\neq\emptyset$, and there is a $k \in A-B$. Consequently,  $\overline{(y_{k}^{(N)},N)}\cdot\beta\equiv_{\alpha}\overline{(y_{k}^{(N)},N)}$ for $\beta\equiv_{\alpha}1$, implying  $\overline{(y_{k}^{(N)},N)}\equiv_{\alpha}\theta$. Let $\pi$ be any permutation on $\aleph_N=\{0,1,\cdots,2^N-1\}$ satisfying that $\pi(i)=i$ when $i\notin A$. There exists a  $\tau\in End(S_c[X])$, such that $\tau(\overline{(y_i^{(N)},N)})=\overline{(y_{\pi(i)}^{(N)},N)}$ for $i\in\aleph_N$ by Lemma \ref{homexist}. Since $\tau(\alpha)=\alpha\equiv_{\alpha}1=\tau(1)$, we have  $\tau(x)\equiv_{\alpha}\tau(y)$ for any $x\equiv_{\alpha}y\in S_c[X]$ by Lemma \ref{homequiv}. Then
$\overline{(y_{\pi(k)}^{(N)},N)}\equiv_{\alpha}\theta$, so $\overline{(y_{i}^{(N)},N)}\equiv_{\alpha}\theta$ for any $i\in A$ by considering $\pi|A$ as an arbitrary permutation on $A$. Consequently, $1\equiv_{\alpha}\alpha\equiv_{\alpha}\theta$, leading to the conclusion that $\equiv_{\alpha}$ is trivial, which is a contradiction.

($\Leftarrow$).Obviously.
\end{proof}

\begin{prop}\label{implicationandinclude}
For any $\alpha,\beta\in S_c[X]$, $\alpha\vdash\beta$ iff $\alpha\rightarrow\beta=1$.
\end{prop}
\begin{proof}
Suppose $\alpha=\overline{(\mathop{\sum^{\circ}}\limits_{i\in A} y_{i}^{(N)},N)}, \beta=\overline{(\mathop{\sum^{\circ}}\limits_{j\in B} y_{j}^{(N)},N)}$. By Proposition \ref{fnmmonic}, $f_{NM}$ is a monomorphism, so $\overline{(\mathop{\sum^{\circ}}\limits_{i\in C} y_{i}^{(N)},N)}=1$ if and only if $C=\aleph_N$ by using the normal forms of $S_c[X_N]$.

($\Rightarrow$). By Proposition \ref{implication}, $A\subseteq B$, $\alpha^c=\overline{(\mathop{\sum^{\circ}}\limits_{i\in \aleph_N-A} y_{i}^{(N)},N)}$, $\alpha^c\circ\beta=\overline{(\mathop{\sum^{\circ}}\limits_{i\in (\aleph_N-A)\cup B} y_{i}^{(N)},N)}$, so $\alpha^c\circ\beta=1$ for $A\subseteq B$, $\alpha\rightarrow\beta=1$.

($\Leftarrow$). Since $\alpha^c\circ\beta=1$, then $(\aleph_N-A)\cup B=\aleph_N$, $A\cap B=A$, hence $A\subseteq B$, $\alpha\vdash\beta$.
\end{proof}
\begin{cor}\label{finitejoininclude}
  Let $\Gamma=\{\alpha_1,\alpha_2,\cdots,\alpha_k\}\subseteq S_c[X]$, and $\alpha_i=\overline{(\mathop{\sum^{\circ}}\limits_{j\in B_i} y_{j}^{(N)},N)}$, $\beta=\overline{(\mathop{\sum^{\circ}}\limits_{j\in B} y_{j}^{(N)},N)}$. Then $\Gamma\vdash\beta$ iff $\bigcap\limits_{i=1}^kB_i\subseteq B$
\end{cor}
\begin{cor}\label{infinitejioninclude}
Let $\Gamma=\{\alpha_i\}_{i \in I}\subseteq S_c[X]$, and $\alpha_i=\overline{(\mathop{\sum^{\circ}}\limits_{j\in B_i^{(N_i)}} y_{j}^{(N_i)},N_i)}$, $\beta=\overline{(\mathop{\sum^{\circ}}\limits_{j\in B^{(M)}} y_{j}^{(M)},M)}$. Then $\Gamma\vdash\beta$ iff  there exist a finite subset $J\subseteq I$ and an  $N\in \mathbb{N}$ such that $\alpha_i=\overline{(\mathop{\sum^{\circ}}\limits_{j\in B_i^{(N)}} y_{j}^{(N)},N)}$ for any $i\in J$, $\beta=\overline{(\mathop{\sum^{\circ}}\limits_{j\in B^{(N)}} y_{j}^{(N)},N)}$,  and $\bigcap\limits_{i\in J}B_i^{(N)}\subseteq B^{(N)}$.
\end{cor}
Denote $A_i=\{2^{i-1}+\sum\limits_{j=1,j\neq i}^kl_j2^{j-1}|l_j=0,1\ and\ k\in\mathbb{N}^+\}$, $A_i^c=\{\sum\limits_{j=1,j\neq i}^kl_j2^{j-1}|l_j=0,1\ and\ k\in\mathbb{N}^+\}$.

\begin{prop}\label{ktoinfinite}
  Suppose 
   $f(A_1^{(k)},\cdots,A_k^{(k)},A_1^{c(k)},\cdots,A_k^{c(k)})$ and

   $g(A_1^{(k)},\cdots,A_k^{(k)},A_1^{c(k)},\cdots,A_k^{c(k)})$ are two mathematical expressions consisting of $\{A_i^{(k)},A_i^{c(k)}\}$ and set operations $\{\cap,\cup\}$.

   $f(A_1^{(k)},\cdots,A_k^{(k)},A_1^{c(k)},\cdots,A_k^{c(k)})\subseteq g(A_1^{(k)},\cdots,A_k^{(k)},A_1^{c(k)},\cdots,A_k^{c(k)})$ iff

   $f(A_1,\cdots,A_k,A_1^c,\cdots,A_k^c)\subseteq g(A_1,\cdots,A_k,A_1^c,\cdots,A_k^c)$.
\end{prop}
\begin{proof}
  For any $x\in f(A_1,\cdots,A_k,A_1^c,\cdots,A_k^c)$,

  we can prove $x\mod2^k\in f(A_1^{(k)},\cdots,A_k^{(k)},A_1^{c(k)},\cdots,A_k^{c(k)})$ by using  induction on the length of $f$. On the other hand, we can also by using induction on the length of $f$ prove that  $x\in f(A_1,\cdots,A_k,A_1^c,\cdots,A_k^c)$ if $x\mod2^k\in f(A_1^{(k)},\cdots,A_k^{(k)},A_1^{c(k)},\cdots,A_k^{c(k)})$.

\end{proof}

Let $\alpha=U_1\circ U_2\cdots\circ U_m\in S_c[X_N]$, $U_k=x_1^{k_1}\cdot x_2^{k_2}\cdots \cdot x_N^{k_N}$,  $x_l^h$ denote $1,x_l,x_l^c$ when $h=0,1,c$. We denote $A_{\alpha}^{(N)}=\bigcup\limits_{i=1}^m\bigcap\limits_{j=1}^{N}A_j^{i_j(N)}$, where $A_j^{i_j(N)}=\aleph_N,A_j^{(N)},A_j^{c(N)}$ when $i_j=0,1,c$.  We also denote
$A_{\alpha}=\bigcup\limits_{i=1}^m\bigcap\limits_{j=1}^{N}A_j^{i_j}$, where $A_j^{i_j}=\mathbb{N},A_j,A_j^{c}$ when $i_j=0,1,c$.

\begin{thm}\label{mainresult}
  Let $\Gamma=\{\Lambda_1,\cdots,\Lambda_n\}\subseteq S_c[X]$, $\Lambda_i=\overline{(\alpha_i,N_i)},\ \alpha_i\in S_c[X_{N_i}]$, $\Lambda=\overline{(\beta,N)},\ \beta\in S_c[X_N]$. Then $\Gamma\vdash\Lambda$ iff $\bigcap\limits_{i=1}^n A_{\alpha_i}^{(M)}\subseteq A_{\beta}^{(M)}$ for some big enough $M\in \mathbb{N}$, iff $\bigcap\limits_{i=1}^nA_{\alpha_i}\subseteq A_{\beta}$.
\end{thm}
\begin{proof}
  For $i\in \{1,2,\cdots,n\}$, $\alpha_i=\mathop{\sum^{\circ}}\limits_{j\in A_{\alpha_i}^{(N_i)}} y_{j}^{(N_i)}$ by the construction of $A_{\alpha}^{(N)}$, and for any $M\geq N_i$,
  $f_{N_iM}(\alpha_i)=\mathop{\sum^{\circ}}\limits_{j\in A_{\alpha_i}^{(M)}} y_{j}^{(M)}$. Therefore, we can choose an $M\geq N_1,\cdots,N_n,N$, such that $\Lambda_i=\overline{(f_{N_iM}(\alpha_i),M)},\ \Lambda=\overline{(f_{NM}(\beta),M)}$. By Corollary \ref{finitejoininclude}, $\Gamma\vdash\Lambda$ iff $\bigcap\limits_{i=1}^n A_{\alpha_i}^{(M)}\subseteq A_{\beta}^{(M)}$, iff $\bigcap\limits_{i=1}^nA_{\alpha_i}\subseteq A_{\beta}$ by Proposition \ref{ktoinfinite}.
\end{proof}

\begin{cor}\label{iffofinfinite}
  Let $\Gamma=\{\Lambda_i\}_{i\in I}\subseteq S_c[X]$, $\Lambda_i=\overline{(\alpha_i,N_i)},\ \alpha_i\in S_c[X_{N_i}]$, $\Lambda=\overline{(\beta,N)},\ \beta\in S_c[X_N]$. Then $\Gamma\vdash\Lambda$ iff there exists a finite subset $J\subseteq I$  such that $\bigcap\limits_{i\in J}A_{\alpha_i}\subseteq A_{\beta}$.
\end{cor}
\begin{cor}\label{infiniteinclude}
   
  Let $\Gamma=\{\Lambda_i\}_{i\in I}\subseteq S_c[X]$, where $\Lambda_i=\overline{(\alpha_i,N_i)}$ and $\alpha_i\in S_c[X_{N_i}]$. Let $\Lambda=\overline{(\beta,N)}$ with $\beta\in S_c[X_N]$. If $\Gamma\vdash\Lambda$, then $\bigcap\limits_{i\in I}A_{\alpha_i}\subseteq A_{\beta}$.
\end{cor}
\begin{rem}
  The converse of Corollary \ref{infiniteinclude} is not true. Let $\Gamma=\{\Lambda_i\}_{i\in \mathbb{N}^+}\subseteq S_c[X]$, $\Lambda_i=\overline{(x_i,i)}$, $\Lambda=\overline{(x_1\cdot x_1^c,1)}$,  $\bigcap\limits_{i=1}^{\infty}A_{x_i}=\bigcap\limits_{i=1}^{\infty}A_{i}=\emptyset$, and for any $k\in \mathbb{N},\ \bigcap\limits_{i=1}^{k}A_{i}\neq\emptyset $.  But $x_1\cdot x_1^c=\theta$, The converse of Corollary \ref{infiniteinclude} is  false by Corollary \ref{iffofinfinite}.
\end{rem}
\section{Examples}
\begin{exa}

  $\alpha_1:$ $a_1$ is an odd number or $a_2$ is an even number,

  $\alpha_2:$ If $a_1$ is an even number, then $a_3$ and $a_4$ are all even numbers,

  $\alpha_3:$ If $a_4$ is an even number, then $a_2$ is an  even number.

   $\beta:$  At least one of $a_2$ and $a_3$ is an even number.

  Let $x_i:$ $a_i$ is an even number ($i=1,2,3,4$), $\Gamma=\{\alpha_1=x_1\circ x_2,\ \alpha_2=x_1^c\circ (x_3\cdot x_4),\ \alpha_3=x_4^c\circ x_2\},\ \beta=x_2\circ x_3$.

  For $X$ is a finite set, we denote $A_{\alpha}^{(4)}$ as $A_{\alpha}$ without confusion, then $A_{\alpha_1}=A_1\cup A_2,\ A_{\alpha_2}=A_2^c\cup (A_3\cap A_4),\ A_{\alpha_3}=A_4^c\cup A_2, A_{\beta}=A_2\cup A_3$.

  $(A_1\cup A_2)\cap (A_2^c\cup (A_3\cap A_4))\cap (A_4^c\cup A_2)=(A_1\cup A_2)\cap (A_2^c\cup A_3)\cap (A_2^c\cup A_4))\cap (A_4^c\cup A_2)$$\subseteq A_2\cup(A_1\cap A_3)\cup (A_1\cap A_2^c\cap A_4^c)\subseteq A_2\cup A_3 =A_{\beta}$

  Therefore $\Gamma\vdash\beta$.
  \end{exa}
  \begin{exa}

    Consider a criminal case involving four individuals denoted as $P_1, P_2, P_3, P_4$. The following clues pertain to the case:

1. If $P_1$ and $P_2$ are not  perpetrators, then $P_3$ and $P_4$ are not  perpetrators.

2. If $P_3$ and $P_4$ are not  perpetrators, then $P_1$ and $P_2$ are not  perpetrators.

3. If $P_1$ and $P_2$ are  perpetrators, then one and only one of $P_3$ and $P_4$ is a perpetrator.

4. If $P_2$ and $P_3$ are  perpetrators, then either both $P_1$ and $P_4$ are perpetrators, or neither $P_1$ nor $P_4$ is a perpetrator.

5. At least one of $P_1, P_2, P_3, P_4$ is a perpetrator.

The question arises: Who are the perpetrators?

   Let $x_i:$ $P_i$ is a perpetrator ($i=1,2,3,4$), $\Gamma=\{\alpha_1=(x_1^c\cdot x_2^c)^c\circ (x_3^c\cdot x_4^c),\ \alpha_2=(x_3^c\cdot x_4^c)^c\circ (x_1^c\cdot x_2^c),\ \alpha_3=(x_1\cdot x_2)^c\circ ((x_3\circ x_4)\cdot (x_3\cdot x_4)^c),\alpha_4=(x_2\cdot x_3)^c\circ((x_1\cdot x_4)\circ(x_1^c\cdot x_4^c)),\alpha_5=x_1\circ x_2\circ x_3\circ x_4 \}$=$\{\alpha_1=x_1\circ x_2\circ (x_3^c\cdot x_4^c),\ \alpha_2=(x_1^c\cdot x_2^c)\circ x_3\circ x_4,\ \alpha_3=x_1^c\circ x_2^c\circ (x_3\cdot x_4^c)\circ (x_3^c\cdot x_4),\alpha_4=x_2^c\circ x_3^c\circ(x_1\cdot x_4)\circ(x_1^c\cdot x_4^c),\alpha_5=x_1\circ x_2\circ x_3\circ x_4 \}$.

 Then $Q=\bigcap\limits_{i=1}^5A_{\alpha_i}=\{5,6,9,11,13\}$. It is easy to find that $Q$ is not a subset of $A_{i}$ or $A_{i}^c\ (i=1,2,3,4)$, so no one can be  identified as a perpetrator, and no one can be cleared of suspicion. Furthermore, we can get $Q\subseteq A_1\cup A_2=\{1, 2, 3, 5, 6, 7, 9, 10, 11, 13, 14, 15\}$, $Q\subseteq A_3\cup A_4=\{4, 5, 6, 7, 8, 9, 10, 11, 12, 13, 14, 15\}$, and $Q\subseteq A_1^c\cup A_2^c\cup A_3^c=\{0, 1, 2, 3, 4, 5, 6, 8, 9, 10, 11, 12, 13, 14\}$, so there is at least one perpetrator in $P_1$ and $P_2$, at least one perpetrator in $P_3$ and $P_4$, and at least one of $P_1$,$P_2$ and $P_3$ is not a perpetrator.
\end{exa}
\begin{exa}
 A group of people are playing a game together. Each person is randomly given a white or black hat (there is at least one black hat). Everyone can see the color of other people's hats, but they cannot see their own hat. The host says for everyone to guess the color of their own hat based on the colors of the hats they can see, then turns off the lights. If someone thinks they are wearing a black hat, they shout it out. If no one reports, turn on the lights and let everyone observe, then turn off the lights again.
No one shouted in the first round of darkness; after turning on the lights and going into the second round of darkness, there was still no movement; it wasn't until the third round of darkness that someone finally shouted out. How many black hats are there?

For solve this problem, let's first prove a conclusion: If someone sees $n$ black hats and no one reports after the $n$-th round of darkness, then someone will  report in the ($n+1$)-th round. We will prove this conclusion by using induction.

Let $\alpha_n$: someone sees $n$ black hats,

$\beta_n$:  someone report in the $n$-th round.

It is easy to see that $\beta_n\rightarrow \beta_{n+1},\ \alpha_{n+1}\cdot(\alpha_n\rightarrow\beta_{n+1})\cdot\beta_{n+1}^c\rightarrow\beta_{n+2}$ and $\alpha_0\rightarrow\beta_1$, we assume that $\alpha_n\cdot\beta_n^c\rightarrow\beta_{n+1}$, we prove that $\alpha_{n+1}\cdot\beta_{n+1}^c\rightarrow\beta_{n+2}$.

Denote $x_1=\alpha_n,\ x_2=\alpha_{n+1},\ x_3=\beta_n,\ x_4=\beta_{n+1},\ x_5=\beta_{n+2}$. Then $\Gamma=\{x_1^c\circ x_3^c\circ x_4,x_3^c\circ x_4, x_1\cdot x_4^c\circ x_4\circ x_2^c\circ x_5\},\ \beta=x_2^c\circ x_4\circ x_5$.

Hence $Q=(A_1^c\cup A_3\cup A_4)\cap(A_3^c\cup A_4)\cap((A_1\cap A_4^c)\cup(A_2^c\cup A_4\cup A_5)),\ A_{\beta}=A_2^c\cup A_4\cup A_5$, since $(A_1^c\cup A_3\cup A_4)\cap(A_3^c\cup A_4)\cap(A_1\cap A_4^c)=\varnothing$, so $Q\subseteq A_{\beta}$. The conclusion is true for any $n$ by induction.

Therefore  there are at least 3 black hats by the conclusion easily.
\end{exa}
\begin{exa}\cite{EFT2021}
  A group $G$ is a triple $(G,\ \cdot,\ e)$ which satisfies (G1)-(G3):

  (G1) For all $x,y,z$: $(x\cdot y)\cdot z=x\cdot (y\cdot z)$.

  (G2) For all x, $x\cdot e=x$.

  (G3) For every $x$ there is a $y$ such that $x\cdot y=e$.

 For every $x$ is there a $y$ such that $y\cdot x=e$?

 We denote $X_{xyz}: x\cdot y=z$($X_{ijk}: x_i\cdot x_j=x_k$ for a countable group), then  (G1)-(G3) can be rewritten as following:

 (G1) $\prod\limits_{x,y,z,a,b,c}(X_{xyz}\cdot X_{yab}\cdot X_{xbc}\rightarrow  X_{zac})$, $\prod\limits_{x,y,z,a,b,c}(X_{xyz}\cdot X_{yab}\cdot X_{zac}\rightarrow  X_{xbc})$.

 (G2) $\prod\limits_{x}X_{xex}$.

 (G3) $\prod\limits_{x}(\mathop{\sum^{\circ}}\limits_{y}X_{xye})$.

 The question can be denote as (Q):$\prod\limits_{y}(\mathop{\sum^{\circ}}\limits_{x}X_{xye})$.

 For a finite group of order $n$, we denote $e$ as $x_0$ and $X_{ijk}$ as $X_{i+j\times n+k\times n^2+1}$.

 If $n=2$, $A_{G1}\cap A_{G2}\cap A_{G3}=\{105,255\}\subseteq A_{Q}=(A_1\cup A_2)\cap(A_3\cup A_4)$, so (Q) is true. But when $n>2$, the calculation becomes infeasible.
\end{exa}
\section*{\textbf{Acknowledgment}}
This work is supported by the National Social Science Fundation of China (No. 21BJY213),
and also by NSF of Anhui University, China (No. KJ2021A0386)

\end{document}